\documentclass{amsart}
\usepackage{amssymb,amsthm,amsmath,epsfig,latexsym,xypic,supertabular}
\usepackage{calc}
\usepackage{latexsym}
\usepackage{amscd,amssymb,subfigure,hyperref}
\usepackage[arrow,matrix,graph,frame,poly,arc,tips]{xy}
\usepackage{color}
\begin{document}

\newcommand{\mmbox}[1]{\mbox{${#1}$}}
\newcommand{\affine}[1]{\mmbox{{\mathbb A}^{#1}}}
\newcommand{\Ann}[1]{\mmbox{{\rm Ann}({#1})}}
\newcommand{\caps}[3]{\mmbox{{#1}_{#2} \cap \ldots \cap {#1}_{#3}}}
\newcommand{\N}{{\mathbb N}}
\newcommand{\Z}{{\mathbb Z}}
\newcommand{\Q}{{\mathbb Q}}
\newcommand{\R}{{\mathbb R}}
\newcommand{\KK}{{\mathbb K}}
\newcommand{\A}{{\mathcal A}}
\newcommand{\B}{{\mathcal B}}
\newcommand{\OO}{{\mathcal O}}
\newcommand{\C}{{\mathbb C}}
\newcommand{\PP}{{\mathbb P}}
\newcommand{\OS}{{T^d(X,p)}}

\newcommand{\reg}{\mathop{\rm reg}\nolimits}
\newcommand{\charr}{\mathop{\rm char}\nolimits}
\newcommand{\ann}{\mathop{\rm ann}\nolimits}
\newcommand{\gin}{\mathop{\rm gin}\nolimits}
\newcommand{\Tor}{\mathop{\rm Tor}\nolimits}
\newcommand{\Ext}{\mathop{\rm Ext}\nolimits}
\newcommand{\Hom}{\mathop{\rm Hom}\nolimits}
\newcommand{\Sym}{\mathop{\rm Sym}\nolimits}
\newcommand{\im}{\mathop{\rm im}\nolimits}
\newcommand{\rk}{\mathop{\rm rk}\nolimits}
\newcommand{\codim}{\mathop{\rm codim}\nolimits}
\newcommand{\supp}{\mathop{\rm supp}\nolimits}
\newcommand{\coker}{\mathop{\rm coker}\nolimits}
\newcommand{\st}{\mathop{\rm st}\nolimits}
\newcommand{\lk}{\mathop{\rm lk}\nolimits}
\sloppy

\newtheorem{thm}{Theorem}[section]
\newtheorem*{thm*}{Theorem}
\newtheorem{defn}[thm]{Definition}
\newtheorem{prop}[thm]{Proposition}
\newtheorem{pref}[thm]{}
\newtheorem*{prop*}{Proposition}
\newtheorem{conj}[thm]{Conjecture}
\newtheorem{lem}[thm]{Lemma}
\newtheorem{rmk}[thm]{Remark}
\newtheorem{cor}[thm]{Corollary}
\newtheorem{notation}[thm]{Notation}
\newtheorem{exm}[thm]{Example}
\newtheorem{comp}[thm]{Computation}

\newcommand{\msp}{\renewcommand{\arraystretch}{.5}}
\newcommand{\rsp}{\renewcommand{\arraystretch}{1}}

\newenvironment{lmatrix}{\renewcommand{\arraystretch}{.5}\small
  \begin{pmatrix}} {\end{pmatrix}\renewcommand{\arraystretch}{1}}
\newenvironment{llmatrix}{\renewcommand{\arraystretch}{.5}\scriptsize
  \begin{pmatrix}} {\end{pmatrix}\renewcommand{\arraystretch}{1}}
\newenvironment{larray}{\renewcommand{\arraystretch}{.5}\begin{array}}
  {\end{array}\renewcommand{\arraystretch}{1}}

  \newenvironment{changemargin}[2]{%
\begin{list}{}{%
\setlength{\topsep}{0pt}%
\setlength{\leftmargin}{#1}%
\setlength{\rightmargin}{#2}%
\setlength{\listparindent}{\parindent}%
\setlength{\itemindent}{\parindent}%
\setlength{\parsep}{\parskip}%
}%
\item[]}{\end{list}}

\title[Free resolutions and Lefschetz properties]
{Free resolutions and Lefschetz properties of \\
some Artin Gorenstein rings of codimension four}
\author[Nancy Abdallah]{Nancy Abdallah}
\address{Department of Mathematics,
University of Bor\aa{}s,
Bor\aa{}s, Sweden}
\email{\href{mailto:nancy.abdallah@hb.se}{nancy.abdallah@hb.se}}
\urladdr{\href{https://www.hb.se/en/shortcuts/contact/employee/NAAB}%
{https://www.hb.se/en/shortcuts/contact/employee/NAAB}}

\author[Hal Schenck]{Hal Schenck}
\thanks{Schenck supported by NSF 2006410\vskip .01in
This article is part of the volume titled ``Computational Algebra and Geometry: A special issue in memory \vskip .01in and honor of Agnes Szanto". Agnes was a superb mathematician and better friend. She will be greatly missed.}
\address{Department of Mathematics,
Auburn University, Auburn, AL 36849}
\email{\href{mailto:hks0015@auburn.edu}{hks0015@auburn.edu}}
\urladdr{\href{http://webhome.auburn.edu/~hks0015/}%
{http://webhome.auburn.edu/~hks0015/}}


\begin{abstract}
\noindent In \cite{S2}, Stanley constructs an example of an Artinian Gorenstein (AG) ring $A$ with non-unimodal $H$-vector $(1,13,12,13,1)$. Migliore-Zanello show in \cite{MZ} 
that for regularity $r=4$, Stanley's example has the smallest possible codimension $c$ for an AG ring with non-unimodal $H$-vector. 

The weak Lefschetz property (WLP) has been much studied for AG rings; it is easy to show that an AG ring with non-unimodal $H$-vector fails to have WLP. In codimension $c=3$ it is conjectured that all AG rings have WLP. For $c=4$, Gondim shows in \cite{G} that WLP always holds for $r \le 4$ and gives a family where WLP fails for any $r \ge 7$, building on Ikeda's example \cite{Ikeda} of failure for $r=5$. In this note we study the minimal free resolution of $A$ and relation to Lefschetz properties (both weak and strong) and Jordan type for $c=4$ and $r \le 6$.
\vskip -.2in
\end{abstract}
\vskip -.2in
\maketitle
\vskip -.2in
\renewcommand{\thethm}{\thesection.\arabic{thm}}
\setcounter{thm}{0}
\vskip -.5in
\section{Introduction}
Let $S=\KK[x_1, \ldots, x_n]$ be a standard graded polynomial ring over a field $\KK$, $I$ a nondegenerate (containing no linear form) homogeneous ideal, and $A = S/I$. The ring $A$ is Gorenstein if it is Cohen-Macaulay and the canonical module $\omega_A$ is isomorphic to a shift of $A$:
\[
\omega_A = \Ext^c_S(A,S)(-n) \simeq A(r)
\]
where $I$ has codimension $c$ and regularity $r$. The graded Betti numbers
\[
b_{i,j}(A) = \dim_{\KK} \Tor_i^S(A,\KK)_j
\]
satisfy a certain symmetry when $A$ is Gorenstein:
\begin{equation}\label{Gsym}
b_{i,j}=b_{i,j}(A) = b_{c-i, r+n-j}(A).
\end{equation}
As $A$ is Cohen-Macaulay, for a linear system of parameters $L$, $b_{i,j}(A)= \dim_{\KK} \Tor_i^{S/L}(A/L,\KK)_j$, and henceforth we assume that $A$ is an Artin Gorenstein (AG) ring. For an AG ring $c=n$, so Equation~\ref{Gsym} becomes $b_{i,j}(A) = b_{n-i, r+n-j}(A)$. Macaulay's famed apolarity theorem \cite{M} shows that any AG ring arises as the inverse system of homogeneous polynomial $F$: there is an apolarity pairing obtained by defining a ring $R=\KK[y_1,\ldots, y_n]$,  and letting $S$ act on $R$ by differentiation: 
\[
x_i(y_j) =\frac{\partial}{\partial(y_i)}(y_j) = \delta_{ij}.
\]
If we define $I_F = \ann_S(F)$ for a homogeneous polynomial $F$ of degree $r$, then Macaulay shows that $S/I_F$ is Gorenstein of regularity and socle degree both $r$; and furthermore that every AG ring arises in this way, with the caveat that in positive characteristic it is necessary to use divided powers. In this note, we investigate Lefschetz properties, which are known (e.g. \cite{BMMNZ2}, \cite{MMRN}) to depend on characteristic. A main tool we employ is the technique of generic initial ideals, which require an infinite ground field \cite{green}, so we assume throughout that $\charr(\KK)=0$.
\subsection{Minimal Free Resolutions}
The Hilbert Syzygy Theorem \cite{e} guarantees that any finitely generated $\Z$-graded $S$-module $M$ has a {\em minimal graded finite free resolution}: an exact sequence 
\begin{equation}\label{FFR}
0 \longrightarrow F_i \longrightarrow F_{i-1} \longrightarrow \cdots \longrightarrow F_0 \longrightarrow M \longrightarrow 0, 
\end{equation}
where $i \le n$ and $F_i \simeq \oplus_{j} S(-j)^{b_{i,j}}$ with $b_{i,j} \in \Z$; in particular  $\dim_{\KK}\Tor_i^S(M, \KK)_j = b_{i,j}$. This data is compactly encoded in the {\em betti table} \cite{e}: 
an array whose entry in position $(i,j)$ (reading over and down) is $b_{i,i+j}$. The reason for this odd indexing is that the index of the bottom row of the betti table encodes the regularity of $M$. 
\begin{exm}
For $F=y_1y_2y_3 \in \KK[y_1,y_2,y_3]$, we have $I_F = \langle x_1^2,x_2^2,x_3^2 \rangle$ and the minimal free resolution is given by the Koszul complex
\[
0 \longrightarrow S(-6) \longrightarrow S(-4)^3 \longrightarrow S(-2)^3 \longrightarrow S \longrightarrow S/I_F \longrightarrow 0, 
\]
which in betti table notation is written as
\begin{center}
	{\scriptsize \begin{verbatim}
			+--------------+
			|       0 1 2 3|
			|total: 1 3 3 1|
			|    0: 1 . . .|
			|    1: . 3 . .|
			|    2: . . 3 .|
			|    3: . . . 1|
			+--------------+
		\end{verbatim}
	}
\end{center}
\end{exm}
\subsection{Lefschetz Properties}
\label{sec:one}Lefschetz properties are ubiquitous in algebra, combinatorics, geometry, and topology. In the setting of commutative algebra, we have
\begin{defn}\label{LefDef}
An Artinian $\Z$-graded ring $A=S/I$ has \begin{enumerate}
	\item  the {\em Weak Lefschetz Property} $($WLP$)$ if there is 
	an $\ell \in S_1$ such that for all $i$, the multiplication map 
	$\mu_{\ell}: A_i \longrightarrow A_{i+1}$ has maximum rank; if not, we say that $A$ fails WLP in degree $i$.
	\item  the {\em Strong Lefschetz Property} $($SLP$)$ if there is 
	an $\ell \in S_1$ such that for all $i$ and $k$, the multiplication map 
	$\mu_{\ell^k}: A_i \longrightarrow A_{i+k}$ has maximum rank; if not we say that $A$ fails SLP in degree $i$. 	
	 
\end{enumerate}
\end{defn}
The set of elements $\ell\in S_1$ with the property that the 
multiplication map $\mu_\ell$ has maximum rank is a (possibly empty) 
Zariski open set in $S_1$, so existence of the Lefschetz 
element $\ell$ in Definition~\ref{LefDef} is equivalent to requiring that multiplication by a general linear form in $S_1$ has full rank in every degree. It is also clear from Definition~\ref{LefDef} that if $A$ has SLP then $A$ has WLP. SLP always holds for $r$=2, and Proposition 3.15 of \cite{HMMNWW} proves that SLP always holds for $c\le 2$ when char$(\KK)=0$, so we focus on $c,r \ge 3$. The simplest AG rings are complete intersections (CI),  and Theorem 2.3 of \cite{HMNW} shows that for $c=3$  a CI always has WLP.
For general $c$ this remains an open question; Boij-Migliore-Mir\'o--Roig-Nagel-Zanello make the following 
\begin{conj}\label{codim3conj}  \cite{BMMNZ2} For $c=3$ and char$(\KK)=0$ an AG ring always has WLP.
\end{conj}
Despite extensive work, Conjecture~\ref{codim3conj} remains open. It is an easy exercise to show that WLP cannot hold for an AG ring with non-unimodal $H$-vector. Migliore-Zanello \cite{MZ} note in Remark 3.2 that in socle degree $4$ the example of Stanley \cite{S2} has the smallest possible codimension, in particular, $c=13$. Hence one might hope that WLP holds for AG rings with small values of $c$ and $r$, and Theorem 3.1 of \cite{G} shows that SLP (hence WLP) always holds for $c=4$ when $r \le 4$. 

We explore the connection between WLP and free resolutions. When $c=4=r$, the possible betti tables of AG rings are determined in \cite{SSY}.  We prove that there are three betti tables possible for an AG ring with $c=4$ and $r=3$. For $c=4$ and $r=5$ we make a conjecture concerning the connection of WLP and the minimal free resolution of $A$. For background on inverse systems and free resolutions, we refer to \cite{e}, and for Lefschetz properties and Jordan type we refer to \cite{HMMNWW}.

\subsection{Results of this paper}
Theorem 2.2 of \cite{SSY} proves that  there are 16 possible betti tables for an AG algebra with $c=4 =r$. The stratification of the parameter space $\PP^{34}$ of quaternary quartics by betti table is described in \S 6 of \cite{KKRSSY}, which notes that for an AG algebra $A$ with $c=4$ and $r=3$ there are only 3 possible betti tables. In \S 2 we prove this assertion, which is non-trivial. 
\begin{thm}\label{43thm}
An AG ring $A$ with $c=4$ and $r=3$ has betti table in the list below:

\begin{center}
{\scriptsize \begin{verbatim}
   
     +-----------------+-----------------+-----------------+
     |       0 1  2 3 4|       0 1  2 3 4|       0 1  2 3 4| 
     |total: 1 6 10 6 1|total: 1 7 12 7 1|total: 1 9 16 9 1|
     |    0: 1 .  . . .|    0: 1 .  . . .|    0: 1 .  . . .|
     |    1: . 6  5 . .|    1: . 6  6 1 .|    1: . 6  8 3 .|
     |    2: . .  5 6 .|    2: . 1  6 6 .|    2: . 3  8 6 .|
     |    3: . .  . . 1|    3: . .  . . 1|    3: . .  . . 1|
     +-----------------+-----------------+-----------------+
\end{verbatim}
}
\end{center}
\end{thm}

\noindent The classification in \cite{SSY} uses the theorems of Macaulay and Gotzmann as the main tools. In contrast, to prove Theorem~\ref{43thm}, we need to analyze certain Groebner strata of the Hilbert scheme. To do this, we use Schreyer's algorithm \cite{FOS} for computing syzygies, described in \S 15.5 of \cite{e}.

The key point is showing that certain maps in the Schreyer resolution must have full rank. This implies that an AG algebra whose betti table has top row {\tt 6 7 2} must have degree two component $I_2$ which is not saturated, and an argument with $\Ext$ modules then shows the betti table below cannot occur for an AG algebra. 
\begin{center}
{\scriptsize \begin{verbatim}
      +-----------------+
      |       0 1 2 3 4 |
      |total: 1 8 14 8 1|
      |    0: 1 .  . . .|
(3)   |    1: . 6  7 2 .|
      |    2: . 2  7 6 .|
      |    3: . .  . . 1|
      +-----------------+
      \end{verbatim}
      }
      \end{center}
    \noindent If we stay in the case of $r=3$ but increase $c$ from $4$ to $5$, then it is easy to find examples where WLP fails--this occurs (and is easy to show) when $I_2$ consists of the $2 \times 2$ minors of a $2 \times 5$ matrix. As noted, for $c=4=r$ WLP always holds, and for $c=4$ and $r=5$ Ikeda \cite{Ikeda} describes an AG algebra $A$ with $H(A) = (1,4,10,10,4,1)$ which fails to have WLP. Our computations indicate that the following is true:
    \begin{conj}\label{conjWLPbetti}
For an AG ring with $c=4$ and $r=5$ there are 36 possible betti tables (see \S 3.2). SLP (and hence WLP) holds for any AG ring of this type with betti table not appearing in Theorem \ref{54thm} below.
For an AG ring with $c=4$ and $r=6$, WLP always holds.
\end{conj}
\noindent The next theorem provides some evidence for Conjecture~\ref{conjWLPbetti}:
\begin{thm}\label{54thm}
For an AG algebra $A$ with $c=4$ and $r=5$ having betti table in the list below, WLP is not determined by the betti table. 

\begin{center}
{\scriptsize \begin{verbatim}

      +-----------------+-------------------+-------------------+-------------------+
      |       0 1  2 3 4|       0  1  2  3 4|       0  1  2  3 4|       0  1  2  3 4|
      |total: 1 9 16 9 1|total: 1 11 20 11 1|total: 1 13 24 13 1|total: 1 16 30 16 1|
      |    0: 1 .  . . .|    0: 1  .  .  . .|    0: 1  .  .  . .|    0: 1  .  .  . .|
      |    1: . 3  2 . .|    1: .  2  1  . .|    1: .  1  .  . .|    1: .  .  .  . .|
      |    2: . 3  6 3 .|    2: .  5  9  4 .|    2: .  7 12  5 .|    2: . 10 15  6 .|
      |    3: . 3  6 3 .|    3: .  4  9  5 .|    3: .  5 12  7 .|    3: .  6 15 10 .|
      |    4: . .  2 3 .|    4: .  .  1  2 .|    4: .  .  .  1 .|    4: .  .  .  . .|
      |    5: . .  . . 1|    5: .  .  .  . 1|    5: .  .  .  . 1|    5: .  .  .  . 1|
      +-----------------+-------------------+-------------------+-------------------+
\end{verbatim}
}
\end{center}
\end{thm}

\noindent For each betti table above, in \S 3 we give an example where WLP holds, and an example where WLP fails. In \S 3.3 we show that SLP can fail for $c=4$ and $r=6$; WLP is unknown in this case.
\pagebreak

\subsection{Computational Methods}For the theorems appearing in \cite{KKRSSY} and \cite{SSY}, and the results and conjecture in \S 1.3 above, evidence was provided by computing in {\tt Macaulay2} the inverse system of all polynomials containing up to four monomial terms, but with all coefficients either zero or one. These computations were made first over $\Z/p$ for small primes $p$, and subsequently over $\Q$. 

We found it surprising that in the cases we considered, all possible betti tables could be generated by polynomials with a small number of monomial terms and simple coefficients, to some extent this is probably due to the strong constraints imposed by the theorems of Macaulay and Gotzmann, combined with the Gorenstein condition. Of course, for codimension 4 and regularity 5, the fact that the list of betti tables is complete is the content of Conjecture~\ref{conjWLPbetti}. We expect that as codimension and regularity become large, inverse systems of more complicated polynomials will come into play. As noted earlier, WLP depends on characteristic of the ground field, so in positive characteristic we expect more exotic behavior. 

The computations in \S 3.1 on Jordan type were performed using {\tt Macaulay2} scripts of Mats Boij, which were written to support work currently in progress. 
\section{Codimension four and Regularity three} 
\noindent We quickly review the theorems of Macaulay and Gotzmann (\S 7.2 of \cite{cag}): for a graded algebra $S/I$ with Hilbert function $h_i$, write 
\[
h_i = \binom{a_i}{i} + \binom{a_{i-1}}{i-1}+ \cdots \mbox{ and }h_i^{\langle i \rangle} = \binom{a_i+1}{i+1} + \binom{a_{i-1}+1}{i}+ \cdots, 
\]
\noindent where $ a_i > a_{i-1} > \cdots$. Then we have
\begin{thm}[Macaulay]
In the setting above, 
\[
h_{i+1} \le h_i^{\langle i \rangle}.
\]
\end{thm}
\begin{thm} [Gotzmann] If $I$ is generated in a single degree $t$ and equality holds in Macaulay's formula in the first degree $t$, then 
\[
h_{t+j} = \binom{a_t+j}{t+j} + \binom{a_{t-1}+j-1}{t+j-1}+ \cdots.
\]
\end{thm}

\noindent Suppose that $S=\KK[x_1,x_2,x_3,x_4]$ ($c=4$) and $r=3$. Since the regularity of $A$ is equal to the socle degree and $A$ is Gorenstein, the Hilbert function of $A$ is $(1,4,4,1)$. Let $I_j$ be the degree $j$ component of the graded ideal $I$. $h_2(S/I)=\binom{5}{2}-\dim(I_2)=4$ so $\dim(I_2)=6$. Hence the betti table of $A$ is of the form
\begin{center}
{\scriptsize \begin{verbatim}
		+------------------------+
		|       0   1   2   3   4|
		|total: 1  6+b 2a  6+b  1|
		|    0: 1   .   .   .   .|
		|    1: .   6   a   b   .|
		|    2: .   b   a   6   .|
		|    3: .   .   .   .   1|
		+------------------------+
\end{verbatim}
}
\end{center}
Since $1=h_3(S/I)=\binom{6}{3}-\dim(I_3)=20-(6\cdot 4-a+b)$, we have $b=a-5$ and so in fact the betti table of $A$ is of the form
{\scriptsize \begin{verbatim}
		+------------------------+
		|       0   1   2   3   4|
		|total: 1  1+a 2a  1+a  1|
		|    0: 1   .   .   .   .|
		|    1: .   6   a  a-5  .|
		|    2: .  a-5  a   6   .|
		|    3: .   .   .   .   1|
		+------------------------+
\end{verbatim}
}
\pagebreak

\begin{lem}\label{L1} If $S/I$ is Artin Gorenstein, then the value of $a$ is in $\{5,6,7,8\}.$
\end{lem}

\begin{proof}
We apply the Theorems 2.1 and 2.2 to determine the possible shapes for the betti table. Let $J_2$ be the subideal generated by the quadrics of $I$. Since  $h_i(S/J_2)=h_i(A)$ for $i\leq 2$, $$h_2(S/J_2)=4=\binom{3}{2}+\binom{1}{1} \text{ and } h_2^{<2>}(S/J_2)=5\geq h_3(S/J_2)=20-(6\cdot 4-a),$$ so $a \le 9$. Since $a$ and $b$ are both nonnegative, $a\geq 5$. If $a=9$, then $ h_2^{<2>}(S/J_2)=h_3(S/J_2)$, and Gotzmann's theorem applies, yielding $$h_4(S/J_2)=\binom{3+2}{2+2}+\binom{1+2-1}{2+2-1}=5.$$
On the other hand, since $b_{3,4}(S/J_2) = a-5 = 4$, \[
h_4(S/J_2)=35-(6\cdot 10-9\cdot 4 - b_{2,4}(S/J_2)+4)=5, \] 
where $(6\cdot 10-9\cdot 4 - b_{2,4}(S/J_2)+4)$ is the dimension of the degree 4 component of $J_2$. So $b_{2,4}(S/J_2)=-2$, a contradiction, and $S/I$ Gorenstein implies $a \in \{5,6,7,8\}$.
\end{proof}
To complete the proof of Theorem~\ref{43thm}, we need to prove the betti table of Table (3) cannot occur for an AG ring (examples show the other betti tables exist). As above, if we let $J_2$ denote the quadratic part of the ideal, then Macaulay's theorem shows that $b_{2,4}(S/J_2) \le 2$. And indeed, such ideals can occur, as shown in the following example:
\begin{exm}
Resolutions with betti table top row {\tt 6 7 2} are possible:
{\scriptsize \begin{verbatim}
i3 : minimalBetti ideal(x_1*x_2, x_1*x_3, x_1*x_4, x_2*x_3, x_2*x_4, x_1^2+x_3*x_4);

            0 1 2 3 4
o3 = total: 1 6 9 5 1
         0: 1 . . . .
         1: . 6 7 2 .
         2: . . 2 3 1
\end{verbatim}
}
\end{exm}
\noindent Theorems 2.1 and 2.2 do not suffice to rule out the existence of an AG ring with betti table top row {\tt 6 7 2}. To do this, we make use of generic initial ideals.
\begin{lem}\label{L2}
The betti table of Table (3) cannot occur for an AG ring.
\end{lem}
\begin{proof}
We consider the possible initial ideals for a resolution where the top row of the betti table is {\tt 6 7 2}. We employ {\em generic initial ideals} and Groebner strata, see \cite{green} for an overview, and Appendix B.2 of \cite{KKRSSY} for details on the {\tt QuaternaryQuartics} package, used for the computations below. As in the proof of Proposition 4.23 of \cite{KKRSSY}, for six quadrics in four variables there are only two possible generic initial ideals which are nondegenerate (do not contain a linear form): the quadratic part of $\gin(I)$ is either $(x_1,x_2,x_3)^2$, or contains a monomial of the form $x_ix_4$, which follows since the $\dim_{\KK} I_2 = 6$. The latter case is impossible, because it forces $I_2$ to contain $\{x_1,x_2,x_3,x_4\} \cdot L$ for a linear form $L$, which is clearly inconsistent with the betti table of Table (3): the linear strand would have betti numbers at least {\tt 4 6 4 1}.  As T. Iarrobino pointed out to us, this also follows because $I$ is an ancestor ideal \cite{Ianc} of $I_3$, which would force $L \in I$. So 
\[
\gin(I)_2 = (x_1,x_2,x_3)^2.
\]
\noindent Since there are only 7 linear syzygies on $I_2$, it cannot be the case that the quadrics in $I$ are a Groebner basis; so in particular $\gin(I)$ must contain a cubic. Thanks to the generic change of coordinates this cubic will have lead term $x_1x_4^2$. Now we compute the Groebner stratum for this family, which entails taking the 7 lead monomials and adding on additional terms with parametric coefficients $t_j$. In the {\tt Macaulay2} computation below the ring {\tt U} has variables $x_i$ and $t_j$, {\tt CF} is the cokernel of the ideal with parametric coefficients, and {\tt H} will denote the nonminimal maps. The resolution of $A$ over $S$ arises by specializing values of $t_i$ in the resolution of {\tt CF} over {\tt U}. 
\pagebreak

{\scriptsize \begin{verbatim}
              2               2         2     2
o14 = ideal (x , x x , x x , x , x x , x , x x );
              3   2 3   1 3   2   1 2   1   1 4

i15 : F=groebnerFamily o14

              2                                 2                                       2                 
o15 = ideal (x  + t x x  + t x x  + t x x  + t x , x x  + t x x  + t x x  + t x x  + t x , x x  + t x x  +
              3    1 1 4    2 2 4    3 3 4    4 4   2 3    5 1 4    6 2 4    7 3 4    8 4   0 2    9 0 3  
      ---------------------------------------------------------------------------------------------------------
                              2   2                                     2                                      
      t  x x  + t  x x  + t  x , x  + t  x x  + t  x x  + t  x x  + t  x , x x  + t  x x  + t  x x  + t  x x  +
       10 2 4    11 3 4    12 4   2    13 1 4    14 2 4    15 3 4    16 4   1 2    17 1 4    18 2 4    19 3 4  
      ---------------------------------------------------------------------------------------------------------
          2   2                                     2     2         2         2       3
      t  x , x  + t  x x  + t  x x  + t  x x  + t  x , x x  + t  x x  + t  x x  + t  x )
       20 4   1    21 1 4    22 2 4    23 3 4    24 4   1 4    25 2 4    26 3 4    27 4

i16 : U=ring F;

i17 : T=coefficientRing U;

i18 : netList F_*

      +------------------------------------------+
      | 2                                 2      |
o18 = |x  + t x x  + t x x  + t x x  + t x       |
      | 3    1 1 4    2 2 4    3 3 4    4 4      |
      +------------------------------------------+
      |                                     2    |
      |x x  + t x x  + t x x  + t x x  + t x     |
      | 2 3    5 1 4    6 2 4    7 3 4    8 4    |
      +------------------------------------------+
      |                                        2 |
      |x x  + t x x  + t  x x  + t  x x  + t  x  |
      | 1 3    9 1 4    10 2 4    11 3 4    12 4 |
      +------------------------------------------+
      | 2                                     2  |
      |x  + t  x x  + t  x x  + t  x x  + t  x   |
      | 2    13 1 4    14 2 4    15 3 4    16 4  |
      +------------------------------------------+
      |                                         2|
      |x x  + t  x x  + t  x x  + t  x x  + t  x |
      | 1 2    17 1 4    18 2 4    19 3 4    20 4|
      +------------------------------------------+
      | 2                                     2  |
      |x  + t  x x  + t  x x  + t  x x  + t  x   |
      | 1    21 1 4    22 2 4    23 3 4    24 4  |
      +------------------------------------------+
      |   2         2         2       3          |
      |x x  + t  x x  + t  x x  + t  x           |
      | 1 4    25 2 4    26 3 4    27 4          |
      +------------------------------------------+

i21 : (CF,H)=nonminimalMaps F;

i22 : U=ring CF;

i23 : CF
            1      7      14      11      3
o23 =      U  <-- U  <-- U   <-- U   <-- U
 
o23 : ChainComplex

i24 : betti(CF, Weights=>{1})

             0 1  2  3 4
o24 = total: 1 7 14 11 3
          0: 1 .  .  . .
          1: . 6  8  3 .
          2: . 1  3  6 3
          3: . .  3  2 .

o24 : BettiTally

i25 : isHomogeneous CF

o25 = true

i26 : keys H

o26 = {(3, 4), (3, 5), (4, 6), (2, 3)}

i27 : M23=H#(2,3)

o27 = {3} | -t_15t_1+t_7t_5-t_14t_5+t_17t_5+t_6t_13-t_9t_13 -t_7t_1+t_17t_1+t_3t_5-t_6t_5-t_9t_5+t_2t_13
      ---------------------------------------------------------------------------------------------------------
      t_16-t_9t_15-t_14t_17+t_17^2+t_19t_5+t_18t_13-t_21t_13 t_8-t_7t_9-t_6t_17+t_9t_17+t_19t_1+t_18t_5-t_21t_5
      ---------------------------------------------------------------------------------------------------------
      t_8-t_7t_9-t_6t_17+t_9t_17+t_11t_5-t_21t_5+t_10t_13 t_4-t_3t_9+t_9^2-t_2t_17+t_11t_1-t_21t_1+t_10t_5
      ---------------------------------------------------------------------------------------------------------
      t_20-t_19t_9-t_18t_17+t_23t_5+t_22t_13 t_12-t_11t_9-t_10t_17+t_23t_1+t_22t_5 |

              1       8
o27 : Matrix T  <--- T

i28 : M34=H#(3,4)

o28 = {4} | t_6t_7-t_7t_9-t_2t_15-t_6t_17+t_9t_17+t_19t_1+2t_18t_5-t_21t_5-t_10t_13         
      {4} | t_7^2-t_7t_14-t_3t_15+t_6t_15+t_9t_15+t_14t_17-t_17^2-t_11t_13-t_18t_13+t_21t_13
      {4} | -t_15t_1+t_7t_5-t_14t_5+t_17t_5+t_6t_13-t_9t_13                                 
      ---------------------------------------------------------------------------------------------------------
      t_3t_6-t_6^2-t_2t_7-t_3t_9+t_9^2+t_2t_14-t_2t_17+t_11t_1+t_18t_1-t_21t_1
      -t_6t_7+t_7t_9+t_2t_15+t_6t_17-t_9t_17+t_19t_1-2t_11t_5+t_21t_5-t_10t_13
      -t_7t_1+t_17t_1+t_3t_5-t_6t_5-t_9t_5+t_2t_13                            
      ---------------------------------------------------------------------------------------------------------
      -t_19t_2+t_11t_6-t_18t_6-t_11t_9+t_18t_9+t_10t_14-2t_10t_17+t_23t_1+t_22t_5   |
      -t_19t_3+t_11t_7-t_18t_7+2t_19t_9+t_10t_15-t_11t_17+t_18t_17-t_23t_5-t_22t_13 |
      -t_19t_1+t_11t_5-t_18t_5+t_10t_13                                             |

              3       3
o28 : Matrix T  <--- T

i29 : M35=H#(3,5)

o29 = {5} | 1 0 0 0 -t_25 0     |
      {5} | 0 1 1 0 -t_26 -t_25 |
      {5} | 0 0 0 1 0     -t_26 |

              3       6
o29 : Matrix T  <--- T

i30 : M46=H#(4,6)

o30 = {6} | 1 0 -t_25 |
      {6} | 0 1 -t_26 |

              2       3
o30 : Matrix T  <--- T
\end{verbatim}
}

\renewcommand{\baselinestretch}{1.5}
\noindent The key point of the computations above is that the last two maps, denoted {\tt M35} and {\tt M46}, are the maps in the Schreyer resolution of {\tt CF}, which has betti table appearing in line {\tt o24}. The map 
\[
\Tor_3^{\tt U}({\tt CF},\KK)_5 \stackrel{\tt M35}{\longrightarrow}
\Tor_2^{\tt U}({\tt CF},\KK)_5
\]
always has rank 3, as we see from line {\tt o29}, and the map 
\[
\Tor_4^{\tt U}({\tt CF},\KK)_6 \stackrel{\tt M46}{\longrightarrow}
\Tor_3^{\tt U}({\tt CF},\KK)_6
\]
always has rank 2, as we see from line {\tt o30}. In particular, this means those maps {\em always} have maximal cancellation, resulting in the betti table below:

\renewcommand{\baselinestretch}{1}
{\scriptsize \begin{verbatim}
             0 1  2  3 4
o24 = total: 1 7  11 6 1
          0: 1 .  .  . .
          1: . 6  8  3 .
          2: . 1  3  3 1
          3: . .  .  . .
\end{verbatim}
}
In order to specialize to top row {\tt 6 7 2} there will also be cancellation in the remaining two maps {\tt M23} and {\tt M34}--on the locus with top row {\tt 6 7 2} the maps {\tt M23} and {\tt M34} both have rank one. The key point is that no cancellation is possible for $\Tor_4^{\tt U}({\tt CF},\KK)_6$, which is nonzero, and so also $\Ext^4_S(S/I_2,S)_6 \ne 0$. Applying $\Hom_S(\bullet, S)$ to the short exact sequence 
\[
0 \longrightarrow I/I_2 \longrightarrow S/I_2 \longrightarrow S/I \longrightarrow 0
\]
yields a long exact sequence in $\Ext$, which in degree 6 terminates as below:
\[
\cdots \rightarrow \Ext^3_S(I/I_2,S)_6 \rightarrow \Ext^4_S(S/I,S)_6  \rightarrow \Ext^4_S(S/I_2,S)_6 \rightarrow \Ext^4_S(I/I_2,S)_6  \rightarrow 0.
\]
Because $I/I_2$ is generated in degree $3$, this forces $\Ext^4_S(I/I_2,S)$ to be generated in degree at least seven. Our computation above shows that $\Ext^4_S(S/I_2,S)_6 \ne 0$, and therefore $\Ext^4_S(S/I,S)_6 \ne 0$, which is inconsistent with Table (3). Hence, Table (3) cannot occur for an AG ring.\end{proof}
\noindent The remaining three betti tables occuring in Theorem~\ref{43thm} may be obtained from the inverse systems of (respectively) the polynomials $y_1^2y_2+y_2y_3y_4$, $y_1^3+y_2y_3y_4$, and $y_1^2y_2+y_3^2y_4$. 
\section{Codimension four and Regularity five}
\noindent We begin by proving Theorem~\ref{54thm}:
\begin{proof} For each betti table, we give examples of ideals satisfying the theorem.
\vskip -.15in
 {\scriptsize \begin{verbatim}
      +-----------------+
      |       0 1 2  3 4|
      |total: 1 9 16 9 1|
      |    0: 1 .  . . .|
      |    1: . 3  2 . .|
      |    2: . 3  6 3 .|
      |    3: . 3  6 3 .|
      |    4: . .  2 3 .|
      |    5: . .  . . 1|
      +-----------------+
      \end{verbatim}
      }
        \vskip -.15in
\noindent ${\bullet}$ WLP holds: $\langle  x_{3}x_{4}, x_{2}x_{4},x_{3}^{2},x_{4}^{3},x_{2}^{2}x_{3}-x_{1}x_{4}^{2}, x_{2}^{3}, x_{1}^{3}x_{3}, x_{1}^{3}x_{2}, x_{1}^{4} \rangle$. \newline

\noindent ${\bullet}$ WLP fails: $\;\langle x_{4}^{2}, x_{3}x_{4}, x_{3}^{2}, x_{2}^{2}x_{4}, x_{2}^{2}x_{3}-x_{1}^{2}x_{4}, x_{1}^{2}x_{3}, x_{2}^{4}, x_{1}^{2}x_{2}^{2}, x_{1}^{4} \rangle$.
\newline 
{\scriptsize \begin{verbatim}
      +--------------------+
      |       0 1  2  3  4 |
      |total: 1 11 20 11 1 |
      |    0: 1 .  .  .  . |
      |    1: . 2  1  .  . |
      |    2: . 5  9  4  . |
      |    3: . 4  9  5  . |
      |    4: . .  1  2  . |
      |    5: . .  .  .  1 |
      +--------------------+
      \end{verbatim}
      }
\noindent ${\bullet}$ WLP holds: $\langle  x_{4}^{2},x_{3}x_{4},x_{3}^{3},x_{2}x_{3}^{2}-x_{2}^{2}x_{4},x_{1}x_{3}^{2}-x_{1}x_{2}x_{4}+x_{2}^{2}x_{4},x_{2}^{2}x_{3},\\
\mbox{     }\;\; \;\;\;\;\;\;\;\;\;\;\;\;\;\;\; \;\;\;\;\;  x_{2}^{3},x_{1}^{3}x_{4}-x_{1}^{2}x_{2}x_{4}+x_{1}x_{2}^{2}x_{4},x_{1}^{3}x_{3},x_{1}^{3}x_{2}-x_{1}^{2}x_{2}^{2},x_{1}^{4} \rangle$. \newline

\noindent ${\bullet}$ WLP fails: $\langle  x_{1}x_{4},x_{1}^{2},x_{3}x_{4}^{2},x_{2}x_{4}^{2},x_{2}^{2}x_{4},x_{1}x_{3}^{2},x_{1}x_{2}^{2}-x_{3}^{2}x_{4},x_{3}^{4},x_{2}x_{3}^{3}-x_{4}^{4},x_{2}^{2}x_{3}^{2},x_{2}^{4} \rangle$.

{\scriptsize \begin{verbatim}
      +--------------------+
      |       0 1  2  3  4 |
      |total: 1 13 24 13 1 |
      |    0: 1 .  .  .  . |
      |    1: . 1  .  .  . |
      |    2: . 7  12 5  . |
      |    3: . 5  12 7  . |
      |    4: . .  .  1  . |
      |    5: . .  .  .  1 |
      +--------------------+
      \end{verbatim}
      }
      \vskip -.1in
\noindent ${\bullet}$ WLP holds: $\langle x_{4}^{2},x_{3}^{2}x_{4},x_{2}^{2}x_{4},x_{3}^{3},x_{2}x_{3}^{2},x_{1}x_{3}^{2}-x_{2}x_{3}x_{4},x_{2}^{2}x_{3},x_{2}^{3}-x_{1}x_{3}x_{4},\\
\mbox{     }\;\; \;\;\;\;\;\;\;\;\;\;\;\;\;\;\; \;\;\;\;\;  x_{1}^{3}x_{4},x_{1}^{3}x_{3}-x_{1}^{2}x_{2}x_{4},x_{1}^{2}x_{2}^{2},x_{1}^{3}x_{2},x_{1}^{4}\ \rangle$. \newline

\noindent ${\bullet}$ WLP fails: $\langle x_{3}^{2},x_{4}^{3},x_{3}x_{4}^{2},x_{2}x_{3}x_{4},x_{1}x_{3}x_{4}-x_{2}x_{4}^{2},x_{2}^{2}x_{4},x_{2}^{2}x_{3}-x_{1}x_{4}^{2},\\
\mbox{     }\;\; \;\;\;\;\;\;\;\;\;\;\;\;\;\;\; \;\;\;\;\;  x_{1}^{2}x_{3}-x_{1}x_{2}x_{4},x_{2}^{4},x_{1}x_{2}^{3}-x_{1}^{3}x_{4},x_{1}^{2}x_{2}^{2},x_{1}^{3}x_{2},x_{1}^{4} \rangle$.
\vskip -.2in
{\scriptsize \begin{verbatim}
      +--------------------+
      |       0 1  2  3  4 |
      |total: 1 16 30 16 1 |
      |    0: 1 .  .  .  . |
      |    1: . .  .  .  . |
      |    2: . 10 15 6  . |
      |    3: . 6  15 10 . |
      |    4: . .  .  .  . |
      |    5: . .  .  .  1 |
      +--------------------+
      \end{verbatim}
      }
      \vskip -.1in
      \noindent ${\bullet}$ WLP holds: for the doubling $($see \cite{KKRSSY}, \S 2.5$)$ of the ideal of $3 \times 3$ minors of a $3 \times 5$ matrix of generic linear forms. The equations are large and unenlightening, so we do not include them here. \newline
      
\noindent ${\bullet}$ WLP fails: $\langle x_{2}x_{4}^{2},x_{1}x_{4}^{2},x_{2}x_{3}x_{4},x_{1}x_{3}x_{4},x_{1}^{2}x_{4},x_{2}x_{3}^{2},x_{1}x_{3}^2,x_2^2x_3,x_{1}^{2}x_{3}-x_{2}^2x_{4},\\
\mbox{     }\;\; \;\;\;\;\;\;\;\;\;\;\;\;\;\;\; \;\;\;\;\;  x_{1}x_{2}^{3}-x_3^3x_4,x_1^2x_2^2,x_{1}^{3}x_{2}-x_{3}^{2}x_{4}^{2},x_{2}^{3}x_{3}-x_{1}^{3}x_{4},x_{1}^{4},x_{2}^{4},x_{3}^{4},x_{4}^{3}\rangle$.\newline

\noindent This is the inverse system for Ikeda's example:
$F=y_1^3y_2y_3+y_1y_2^3y_4+y_3^3y_4^2$. 
\end{proof}
\subsection{Connection to Jordan Type}
One property of Artinian $\KK$-algebras that has generated much investigation is the Jordan type.
\begin{defn}
	Let $A=\bigoplus_{k\geq 0} A_k$ be a graded Artinian algebra and $\ell \in A_1$ a linear form. The Jordan type of $A$ with respect to $\ell$ is the partition of $\dim_\mathbb{K}A$ denoted by $P_\ell=P_{\ell,A}=(p_1,\dots,p_s)$ where $p_1\geq p_2\ge \cdots \ge p_s$ and $p_i$'s are the block sizes in the Jordan canonical form matrix of the multiplication by $\ell$. 
\end{defn}

\begin{prop}[see 3.5 of \cite{HMMNWW} or 2.11 of \cite{IMM}] 
If the Hilbert function of a graded Artinian algebra is unimodal and symmetric, then $A$ has WLP iff the maximal value of $h_i(A)$ is equal to the length of the partition for a generic $\ell$.
\end{prop}
\noindent For the algebras appearing in Theorem~\ref{54thm}, we have the following:
\begin{comp}\label{comp54}
The Jordan decompositions for $S/I$ as in Theorem~\ref{54thm} are
\begin{enumerate}
\item When $I$ has 3 quadrics \newline
        \vskip -.1in
\noindent ${\bullet}$ If WLP holds, the Jordan type is  $\{6, 4, 4, 4, 2, 2, 2\}$.\newline 
\noindent ${\bullet}$ If WLP fails, the Jordan type is $\{6, 4, 4, 4, 2, 2, 1, 1\}$.
\vskip .1in
\item When $I$ has 2 quadrics \newline
        \vskip -.15in
\noindent ${\bullet}$ If WLP holds, the Jordan type is $\{6, 4, 4, 4, 2, 2, 2, 2\}$.\newline
\noindent ${\bullet}$ If WLP fails, the Jordan type is $\{6, 4, 4, 4, 2, 2, 2, 1, 1\}$.\newline

\item  \vskip -.10in When $I$ has 1 quadric \newline
        \vskip -.15in
\noindent ${\bullet}$ If WLP holds, the Jordan type is $\{6, 4, 4, 4, 2, 2, 2, 2, 2\}$.\newline
\noindent ${\bullet}$ If WLP fails, the Jordan type is $\{6, 4, 4, 4, 2, 2, 2, 2, 1, 1\}$.\newline
  \vskip -.15in
\item \vskip -.12inWhen $I$ has no quadrics \newline
        \vskip -.15in
\noindent ${\bullet}$ If WLP holds, the Jordan type is $\{6, 4, 4, 4, 2, 2, 2, 2, 2, 2\}$.\newline
\noindent ${\bullet}$ If WLP fails, the Jordan type is $\{6, 4, 4, 4, 2, 2, 2, 2, 2, 1, 1\}$.
\end{enumerate}
\end{comp}
\vskip -.04in
\noindent Proposition 2.10 of \cite{IMM} shows that $A$ has SLP iff the partition $P_{A,\ell}$ is conjugate to the $H$-vector. Hence the examples in Computation~\ref{comp54} have SLP if they have WLP.  This has been proven recently in \cite{A}. By  \cite{AAISY}, SLP holds for an AG algebra $A$ with $c=3$ if all $h_i(A) \le 6$, which also seems to hold for $A$ with $c=4$ and $r=5$.
\pagebreak

\subsection{Conjecture~\ref{conjWLPbetti}} 
\noindent {\tt Macaulay2} computations suggest that the betti tables possible for an AG algebra with $c=4$ and $r=5$ are those below, which is (part of) the content of Conjecture~\ref{conjWLPbetti}:
\vskip .1in
\begin{changemargin}{0.1cm}{0.1cm}
{\tiny \begin{verbatim}
  +-------------------+-----------------+-------------------+-------------------+-------------------+-------------------+
      |       0 1  2 3 4  |       0 1  2 3 4|       0 1  2 3 4  |       0 1  2 3 4  |       0 1  2 3 4  |       0 1  2 3 4  |
      |total: 1 9 16 9 1  |total: 1 6 10 6 1|total: 1 9 16 9 1  |total: 1 7 12 7 1  |total: 1 9 16 9 1  |total: 1 6 10 6 1  |
      |    0: 1 .  . . .  |    0: 1 .  . . .|    0: 1 .  . . .  |    0: 1 .  . . .  |    0: 1 .  . . .  |    0: 1 .  . . .  |
      |    1: . 6  8 3 .  |    1: . 5  5 . .|    1: . 5  6 2 .  |    1: . 4  3 . .  |    1: . 4  4 1 .  |    1: . 4  2 . .  |
      |    2: . .  . . .  |    2: . .  . 1 .|    2: . 1  2 1 .  |    2: . 1  3 2 .  |    2: . 2  4 2 .  |    2: . .  3 2 .  |
      |    3: . .  . . .  |    3: . 1  . . .|    3: . 1  2 1 .  |    3: . 2  3 1 .  |    3: . 2  4 2 .  |    3: . 2  3 . .  |
      |    4: . 3  8 6 .  |    4: . .  5 5 .|    4: . 2  6 5 .  |    4: . .  3 4 .  |    4: . 1  4 4 .  |    4: . .  2 4 .  |
      |    5: . .  . . 1  |    5: . .  . . 1|    5: . .  . . 1  |    5: . .  . . 1  |    5: . .  . . 1  |    5: . .  . . 1  |
      +-------------------+-----------------+-------------------+-------------------+-------------------+-------------------+
      |       0 1  2 3 4  |       0 1  2 3 4|       0 1  2 3 4  |       0  1  2  3 4|       0  1  2  3 4|       0 1 2 3 4   |
      |total: 1 9 16 9 1  |total: 1 8 14 8 1|total: 1 6 10 6 1  |total: 1 11 20 11 1|total: 1 10 18 10 1|total: 1 4 6 4 1   |
      |    0: 1 .  . . .  |    0: 1 .  . . .|    0: 1 .  . . .  |    0: 1  .  .  . .|    0: 1  .  .  . .|    0: 1 . . . .   |
      |    1: . 3  2 . .  |    1: . 3  1 . .|    1: . 3  2 . .  |    1: .  2  1  . .|    1: .  2  .  . .|    1: . 3 . . .   |
      |    2: . 3  6 3 .  |    2: . 2  6 3 .|    2: . 3  3 . .  |    2: .  5  9  4 .|    2: .  4  9  4 .|    2: . 1 3 . .   |
      |    3: . 3  6 3 .  |    3: . 3  6 2 .|    3: . .  3 3 .  |    3: .  4  9  5 .|    3: .  4  9  4 .|    3: . . 3 1 .   |
      |    4: . .  2 3 .  |    4: . .  1 3 .|    4: . .  2 3 .  |    4: .  .  1  2 .|    4: .  .  .  2 .|    4: . . . 3 .   |
      |    5: . .  . . 1  |    5: . .  . . 1|    5: . .  . . 1  |    5: .  .  .  . 1|    5: .  .  .  . 1|    5: . . . . 1   |
      +-------------------+-----------------+-------------------+-------------------+-------------------+-------------------+
      |       0  1  2  3 4|       0 1  2 3 4|       0  1  2  3 4|       0 1  2 3 4  |       0 1  2 3 4  |       0 1  2 3 4  |
      |total: 1 11 20 11 1|total: 1 7 12 7 1|total: 1 13 24 13 1|total: 1 8 14 8 1  |total: 1 7 12 7 1  |total: 1 7 12 7 1  |
      |    0: 1  .  .  . .|    0: 1 .  . . .|    0: 1  .  .  . .|    0: 1 .  . . .  |    0: 1 .  . . .  |    0: 1 .  . . .  |
      |    1: .  3  3  1 .|    1: . 2  . . .|    1: .  1  .  . .|    1: . 2  1 . .  |    1: . 3  2 . .  |    1: . 4  4 1 .  |
      |    2: .  4  7  3 .|    2: . 4  6 1 .|    2: .  7 12  5 .|    2: . 5  6 1 .  |    2: . 3  4 1 .  |    2: . 2  2 . .  |
      |    3: .  3  7  4 .|    3: . 1  6 4 .|    3: .  5 12  7 .|    3: . 1  6 5 .  |    3: . 1  4 3 .  |    3: . .  2 2 .  |
      |    4: .  1  3  3 .|    4: . .  . 2 .|    4: .  .  .  1 .|    4: . .  1 2 .  |    4: . .  2 3 .  |    4: . 1  4 4 .  |
      |    5: .  .  .  . 1|    5: . .  . . 1|    5: .  .  .  . 1|    5: . .  . . 1  |    5: . .  . . 1  |    5: . .  . . 1  |
      +-------------------+-----------------+-------------------+-------------------+-------------------+-------------------+
      |       0 1  2 3 4  |       0 1  2 3 4|       0 1  2 3 4  |       0  1  2  3 4|       0 1  2 3 4  |       0  1  2  3 4|
      |total: 1 6 10 6 1  |total: 1 9 16 9 1|total: 1 9 16 9 1  |total: 1 10 18 10 1|total: 1 7 12 7 1  |total: 1 11 20 11 1|
      |    0: 1 .  . . .  |    0: 1 .  . . .|    0: 1 .  . . .  |    0: 1  .  .  . .|    0: 1 .  . . .  |    0: 1  .  .  . .|
      |    1: . 3  1 . .  |    1: . 2  1 . .|    1: . 3  3 1 .  |    1: .  1  .  . .|    1: . 2  1 . .  |    1: .  1  .  . .|
      |    2: . 2  4 1 .  |    2: . 5  7 2 .|    2: . 4  5 1 .  |    2: .  7  9  2 .|    2: . 5  5 . .  |    2: .  7 10  3 .|
      |    3: . 1  4 2 .  |    3: . 2  7 5 .|    3: . 1  5 4 .  |    3: .  2  9  7 .|    3: . .  5 5 .  |    3: .  3 10  7 .|
      |    4: . .  1 3 .  |    4: . .  1 2 .|    4: . 1  3 3 .  |    4: .  .  .  1 .|    4: . .  1 2 .  |    4: .  .  .  1 .|
      |    5: . .  . . 1  |    5: . .  . . 1|    5: . .  . . 1  |    5: .  .  .  . 1|    5: . .  . . 1  |    5: .  .  .  . 1|
      +-------------------+-----------------+-------------------+-------------------+-------------------+-------------------+
      |       0 1  2 3 4  |       0 1  2 3 4|       0  1  2  3 4|       0 1  2 3 4  |       0  1  2  3 4|       0  1  2  3 4|
      |total: 1 6 10 6 1  |total: 1 9 16 9 1|total: 1 14 26 14 1|total: 1 8 14 8 1  |total: 1 13 24 13 1|total: 1 12 22 12 1|
      |    0: 1 .  . . .  |    0: 1 .  . . .|    0: 1  .  .  . .|    0: 1 .  . . .  |    0: 1  .  .  . .|    0: 1  .  .  . .|
      |    1: . 2  . . .  |    1: . 1  . . .|    1: .  .  .  . .|    1: . 1  . . .  |    1: .  .  .  . .|    1: .  .  .  . .|
      |    2: . 4  5 . .  |    2: . 7  8 1 .|    2: . 10 13  4 .|    2: . 7  7 . .  |    2: . 10 12  3 .|    2: . 10 11  2 .|
      |    3: . .  5 4 .  |    3: . 1  8 7 .|    3: .  4 13 10 .|    3: . .  7 7 .  |    3: .  3 12 10 .|    3: .  2 11 10 .|
      |    4: . .  . 2 .  |    4: . .  . 1 .|    4: .  .  .  . .|    4: . .  . 1 .  |    4: .  .  .  . .|    4: .  .  .  . .|
      |    5: . .  . . 1  |    5: . .  . . 1|    5: .  .  .  . 1|    5: . .  . . 1  |    5: .  .  .  . 1|    5: .  .  .  . 1|
      +-------------------+-----------------+-------------------+-------------------+-------------------+-------------------+
      |       0  1  2  3 4|       0 1  2 3 4|       0 1  2 3 4  |       0 1  2 3 4  |       0  1  2  3 4|       0  1  2  3 4|
      |total: 1 11 20 11 1|total: 1 7 12 7 1|total: 1 7 12 7 1  |total: 1 8 14 8 1  |total: 1 10 18 10 1|total: 1 16 30 16 1|
      |    0: 1  .  .  . .|    0: 1 .  . . .|    0: 1 .  . . .  |    0: 1 .  . . .  |    0: 1  .  .  . .|    0: 1  .  .  . .|
      |    1: .  .  .  . .|    1: . 5  5 1 .|    1: . 3  . . .  |    1: . 2  . . .  |    1: .  .  .  . .|    1: .  .  .  . .|
      |    2: . 10 10  1 .|    2: . .  1 1 .|    2: . 1  6 3 .  |    2: . 4  7 2 .  |    2: . 10  9  . .|    2: . 10 15  6 .|
      |    3: .  1 10 10 .|    3: . 1  1 . .|    3: . 3  6 1 .  |    3: . 2  7 4 .  |    3: .  .  9 10 .|    3: .  6 15 10 .|
      |    4: .  .  .  . .|    4: . 1  5 5 .|    4: . .  . 3 .  |    4: . .  . 2 .  |    4: .  .  .  . .|    4: .  .  .  . .|
      |    5: .  .  .  . 1|    5: . .  . . 1|    5: . .  . . 1  |    5: . .  . . 1  |    5: .  .  .  . 1|    5: .  .  .  . 1|
      +-------------------+-----------------+-------------------+-------------------+-------------------+-------------------+
      \end{verbatim}
      }
     \end{changemargin}
         
\subsection{WLP and SLP for $c=4$ and $r=6$} The Strong Lefschetz Property does not always hold when $c=4$ and $r=6$: the inverse system of $F=y_1^2y_4(y_1^2y_3+y_2y_4^2)$ has Hilbert function $(1,4,7,8,7,4,1)$ and Jordan type $\{7, 5, 5, 5, 3, 3, 2, 2\}$, so by Proposition 2.10 of \cite{IMM} cannot satisfy SLP. Gondim shows in Theorem 3.8 of \cite{G} that there exist examples where WLP fails for $c=4$ and all $r \ge 7$. For $c=4$ and $r=6$, we found no examples of failure of WLP. Since WLP fails for $r=5$ and $r \ge 7$, it would be interesting to investigate if indeed WLP always holds for $r=6$.

\begin{rmk}
	In \cite{MVW}, Macias Marques, Veliche and Weyman prove that all codimension 4 and regularity 3 AG rings come from the doubling construction. As a consequence they also obtain the betti table classification of Theorem \ref{43thm}.
\end{rmk}
\vskip .05in
\noindent{\bf Acknowledgments} All of the computations in this paper were performed using the {\tt QuaternaryQuartics} \cite{KKRSSY} package of Macaulay2, by Grayson and Stillman, which is available at: {\tt http://www.math.uiuc.edu/Macaulay2/}.
Our collaboration began at the 2019 CIRM workshop ``Lefschetz properties in algebra, geometry, and combinatorics'', organized by A. Dimca, R. Mir\'o-Roig, and J. Vall\`es, and we thank them and CIRM for a great workshop.  We thank Nasrin Altafi, Tony Iarrobino, Juan Migliore, and two anonymous referees for helpful comments.
\vskip -.01in
\bibliographystyle{amsalpha}

\end{document}